\documentclass[a4paper,11pt]{amsart}
\usepackage{amsmath,amssymb,amsfonts,amsthm,calc,color}

\usepackage[dvipsnames]{xcolor}

\usepackage[german, english]{babel}
\usepackage{hyperref}

\newtheorem{lemma}{Lemma}[section]
\newtheorem{theorem}[lemma]{Theorem}
\newtheorem{thm}[lemma]{Theorem}

\theoremstyle{definition}
\newtheorem{definition}[lemma]{Definition}

\newtheorem*{remark}{Remark}

\numberwithin{equation}{section}
\newcommand{\comment}[1]{}

\newcommand{\Z}{{\mathbb Z}}
\newcommand{\R}{{\mathbb R}}

\newcommand{\N}{{\mathbb N}}

\newcommand{\ls}{{\log^\sharp}}

\newcommand{\al}{{\alpha}}

\newcommand{\de}{{\delta}}
\newcommand{\eps}{{\varepsilon}}

\newcommand{\ph}{{\varphi}}

\newcommand{\rh}{{\varrho}}

\newcommand{\ov}[1]{\overline{ #1}}

\newcommand{\Hm}[1]{\leavevmode{\marginpar{\tiny%
$\hbox to 0mm{\hspace*{-0.5mm}$\leftarrow$\hss}%
\vcenter{\vrule depth 0.1mm height 0.1mm width \the\marginparwidth}%
\hbox to 0mm{\hss$\rightarrow$\hspace*{-0.5mm}}$\\\relax\raggedright
#1}}}

\begin{document}

\title[Uniqueness class, stochastic completeness and volume growth]{On the uniqueness class, stochastic completeness and volume growth for graphs}

\author[X. Huang]{Xueping Huang}
\address{Xueping Huang, Department of Mathematics, Nanjing University of Information Science and Technology, 210044 China} \email{hxp@nuist.edu.cn} 

\author[M. Keller]{Matthias Keller}
\address{Matthias Keller, Universit\"at Potsdam, Institut f\"ur Mathematik, 14476  Potsdam, Germany}
\email{matthias.keller@uni-potsdam.de}

\author[M. Schmidt]{Marcel Schmidt}
\address{Marcel Schmidt, Mathematisches Institut \\Friedrich Schiller Universit{\"a}t Jena \\07743 Jena, Germany } \email{schmidt.marcel@uni-jena.de}

\begin{abstract}
In this note we prove an optimal volume growth condition for stochastic completeness of graphs under very mild assumptions. This is realized by proving a uniqueness class criterion for the heat equation which is an analogue to a corresponding result of Grigor'yan  on manifolds. This uniqueness class criterion is shown to hold for graphs that we call globally local, i.e., graphs where we control the jump size  far outside.  The transfer from general graphs to  globally local graphs is then carried out via so called refinements.
\end{abstract}


\maketitle



\section*{Introduction}

In 1980 Azencott \cite{Az74} gave an example of a complete Riemannian manifold on which the Brownian motion has finite lifetime. Such manifolds are referred to as stochastically incomplete and typically are of very large volume growth. On the other hand it was shown that stochastic completeness is guaranteed under certain volume bounds which were improved over the years, see Gaffney \cite{Ga59}, Karp/Li \cite{KarpLi}, 
 Davies \cite{Da92} and Takeda  \cite{Tak89}. An optimal result was obtained by  Grigor'yan \cite{Gri86}
who proved stochastic completeness of the manifold under an integral criterion involving the volume (see also \cite{Gri99}) and he showed by examples that his criterion is sharp. Later, Grigor'yan's result was extended by Sturm \cite{Stu94} to strongly local Dirichlet forms where the phenomenon is referred to as conservativeness and distance balls were considered with respect to a so called intrinsic metric. Indeed, the proof in this more general situation follows in spirit of Grigor'yan. A remarkable feature of  Grigor'yan's proof is that it not only yields stochastic completeness but the crucial inequality (to which we refer as Grigor'yan's inequality) directly implies a uniqueness class statement for the heat equation. Precisely, while stochastic completeness is equivalent to uniqueness of bounded solutions to the heat equation, the uniqueness class statement extends this uniqueness to a class of unbounded solutions which satisfy a certain growth bound.

In recent years the phenomenon of stochastic completeness was intensively studied in the non-local realm of graphs. The enormous interest in this topic was sparked by the PhD thesis \cite{Woj08} and follow up work \cite{Woj11} of Wojciechowski who presented examples of graphs of polynomial volume growth, which are stochastically incomplete. This showed that there is no analogous result to Grigor'yan's in the non-local realm of graphs when one considers volume growth of balls with respect to the combinatorial graph distance.
 However, in view of the work of Sturm \cite{Stu94} for local Dirichlet forms, which uses intrinsic metrics, it seemed promising to consider distance balls with respect to a metric that is adapted to the heat flow on the graph. While such a theory of intrinsic metrics was developed at this time also for non-local (and thus for all regular) Dirichlet forms, this idea was used by Grigor'yan/Huang/Masamune \cite{GHM12} to prove a first result in this direction that guaranteed stochastic completeness of the graph provided an exponential (but not optimal) bound on the volume growth.  Shortly afterwards Grigor'yan's result for manifolds was recovered for graphs using so called intrinsic (or adapted) metrics by Folz \cite{Fol14} and shortly after that an alternative proof was given by Huang \cite{Huang14}.  See also \cite{HS14} for results on the closely related problem of escape rates.

In spirit, the proofs of these results used techniques that relate the non-local graph to a more local object. Specifically, Folz \cite{Fol14} compared the heat flow on the combinatorial graph with a corresponding metric (or quantum) graph and Huang and Shiozawa \cite{HS14} decreased non-locality of the graph by inserting additional vertices in the edges (which probabilistically decreased the jump size of the process). Although this was an enormous breakthrough, there are two aspects in which the results are not completely satisfying -- one of technical the other of structural nature. The technical aspect is that the results were proven under rather restrictive conditions such as local finiteness of the graphs, finite jump size of the metric and uniform lower bounds on the measure. Moreover, the only metrics considered were special path metrics. These restrictions did not inspire much hope that the proof strategies can be carried over to more general jump processes. The second aspect, which may be seen as a shortcoming  of more fundamental nature, is that the proofs do not allow to recover Grigor'yan's inequality. So, as a consequence, one can not deduce corresponding statements about the uniqueness class for the heat equation. Indeed, this is not a shortcoming of the proofs but the inequality simply does not hold for general graphs. In his PhD thesis \cite{Hua11} Huang gave an example of a nontrivial solution to the heat equation on the integer line which showed that the corresponding uniqueness  class statement of Grigor'yan, which directly follows from Grigor'yan's inequality, is already wrong in this simple case of a graph.

The purpose of this paper is to address these two aspects. Firstly, we prove Grigor'yan's criterion for stochastic completeness for graphs under the only assumption that there exists an intrinsic pseudo metric whose distance balls are finite. Secondly,  we address the second aspect mentioned above in a structural way and recover Grigor'yan's inequality for graphs which we call globally local or  GL graphs for short.\footnote{The term GL graphs may be also read as an abbreviation for Grigor'yan-Lenz graphs in honor of our PhD advisors who never ceased to push us towards properly understanding the issue at hand.} These are graphs for which an intrinsic metric exists such that the jump size becomes  small far outside with respect to the distance to a reference point. In this sense they appear to be more and more local on a global scale, i.e., outside of large compact sets. Our proof of Grigor'yan's inequality for globally local graphs follows the strategy of Grigor'yan's original proof. However, the non-locality of the space requires certain crucial modifications. Let us already stress  at this point that our proof does not make use of the discreteness of the space but can be applied to more general jump processes. The proof of the result for stochastic completeness then uses an idea of Huang and Shiozawa \cite{HS14}. We refine a given graph by inserting additional vertices on the edges, so that the resulting graph is globally local, and then employ a stability result for stochastic completeness. In the proof of this stability result we use the discreteness of the space and it is the reason why this paper does not treat general jump processes. 

In summary, we give a new approach to proving an optimal volume growth condition for   stochastic completeness of graphs. It is based on two ingredients: A uniqueness class for the heat equation and a stability result for stochastic completeness. While our proof for the uniqueness class also works for general jump processes, as of yet the stability statement is restricted to the discrete setting. In other words, a suitable stability statement for stochastic completeness for general jump processes combined with our uniqueness class for (globally local) jump processes would lead to an optimal volume growth test for stochastic completeness for all jump processes, a problem that is still open.

\section{Set-up and main result}

Let $ X $ be a discrete countable set and let  $ m:X\to(0,\infty) $.  We  extend  $m$ to a measure on all subsets of $X$ via additivity.  For $1 \leq p \leq \infty$ let $ \ell^{p}(X,m) $ be the canonical real Banach space with norm $ \|\cdot\|_{p} $.
Moreover, we let $C(X)$ be the space of real-valued functions on $X$ and $C_c(X)$ be the space of real-valued functions of finite support.

A graph over $ (X,m) $ is a symmetric map $ b:X\times X\to[0,\infty) $ with zero diagonal and
\begin{align*}
\sum_{y\in X} b(x,y)<\infty,\quad x\in X.
\end{align*}
For $ x,y\in X $ we write $ x\sim y $ whenever $ b(x,y)>0 $. In this case we call $ (x,y) $  an \emph{edge}. The graph is called {\em locally finite} if for each $x \in X$ the number of edges containing $x$ is finite, i.e., if 
$$\sharp \{y \in X \mid b(x,y) > 0\} < \infty.$$

The gradient squared of a function $ f:X\to [0,\infty] $
is defined as
\begin{align*}
|\nabla f|^{2}(x)=\sum_{y\in X}b(x,y)(f(x)-f(y))^{2},\quad x\in X.
\end{align*}
It might take the value $ \infty $ but is finite for  $ f\in C_{c}(X). $

Due to the summability condition on $b$, the quadratic form
\begin{align*}
Q(f)=\frac{1}{2}\sum_{x\in X}|\nabla f  |^{2}(x)
\end{align*}
on $ D(Q) =\ov{C_{c}(X)}^{\|\cdot\|_{Q}}$ with $ \|\cdot\|_{Q} =\sqrt{Q(\cdot)+\|\cdot\|_2^2}$ is a regular Dirichlet form on $\ell^2(X,m)$. The associated positive self-adjoint operator $ L  $ is a restriction of the formal Laplacian
\begin{align*}
\mathcal{L}f(x)=\frac{1}{m(x)}\sum_{y\in X}b(x,y)(f(x)-f(y)),
\end{align*}
which is defined on its domain
$$
\mathcal{F}(X)=\{f:X\to\R\mid \sum_{y\in X}b(x,y)|f(y)|<\infty,x\in X\},
$$
see e.g. \cite{KL10}. If the graph is locally finite, then $\mathcal{F}(X) = C(X)$.

For a given $0 < T \leq \infty$ we say that a function $u:(0,T) \times X \to \R$ is a {\em solution to the heat equation with initial value $u_0 \in C(X)$} if the following conditions are satisfied.
\begin{itemize}
 \item  For every $t > 0$ we have $u_t \in \mathcal{F}(X)$,
 \item for every $x \in X$ the function $t \mapsto u_t(x)$ is continuously differentiable and $u_t(x) \to u_0(x)$, for $t \to 0+$, 
 \item $u$ solves the heat equation, i.e., 
 $$\partial_t u = -\mathcal{L}u \text{ on } (0,T) \times X.$$
\end{itemize}
Here and also below we write $u_t$ for the function $X \to \R, \, x \mapsto u(t,x)$.

Since $ Q $ is a Dirichlet form, the semigroup $ e^{-tL} $, $ t > 0 $, extends to $ \ell^{p}(X,m) $, $ p\in[1,\infty]$.  For $f\in\ell^{p}(X,m)$  the function $u:(0,\infty) \times X \to \R$ defined by $u_t = e^{-tL}f $  is a solution to the heat equation with initial value $f$, see e.g. \cite{KL12}.

The graph $ b $ over $ (X,m) $ is called \emph{stochastically complete} if
\begin{align*}
e^{-tL} 1=1
\end{align*}
for some (all) $ t>0 $, where $ 1 $ is the function constantly equal to $1$. Note that this definition depends on both $b$ and $m$. It corresponds to the property that the associated Markov process has infinite lifetime. Stochastic completeness is related to uniqueness of bounded solutions to the heat equation. More precisely, a graph $b$ over $(X,m)$ is stochastically complete if and only if any bounded (in space and time) solution to the heat equation with initial value $0$ vanishes, see \cite{KL12}.

We call a (pseudo) metric $ d $ on $ X $ an \emph{intrinsic (pseudo) metric} (for the graph $b$ over $(X,m)$) if for all $ x\in X $ 
\begin{align*}
\sum_{y \in X} b(x,y) d(x,y)^2  = |\nabla d(x,\cdot)|^{2}\leq m(x).
\end{align*}

We denote the distance balls for a pseudo metric about a  reference vertex $o\in X $  with radius $ r \ge 0 $ by
\begin{align*}
B_{r}=\{y\in X\mid d(o,y)\leq r\}.
\end{align*}
The reference vertex $o$ is fixed throughout the paper.

With these notions we can formulate our main theorem. It is the precise analogue to Grigor'yan's optimal volume growth criterion on Riemannian manifolds, cf. \cite{Gri86}. We let $\ls = \max\{\log,1\}$.

\begin{thm}\label{thm:main}
	Let $ b $ be a graph over $ (X,m) $  and let $ d $ be an intrinsic pseudo metric with finite distance balls. If
	\begin{align*}
	\int_{1}^{\infty}\frac{ r}{\ls({m(B_{r})})}dr=\infty,
	\end{align*}
	then  the graph is stochastically complete.
\end{thm}

Note that the volume growth condition of the theorem is independent of the choice of the reference point.  The lower bound  $1$ in the interval of integration could also be replaced by any positive number.  Hence, this criterion only depends on the growth of $\log m(B_r)$ for large $r$.

To prove the result we first show a uniqueness class theorem for the heat equation on graphs where the jump size becomes sufficiently small far outside. We refer to these graphs as \emph{globally local} or \emph{GL graphs}
with respect to a function $ f $. As mentioned above stochastic completeness is equivalent to uniqueness of bounded solutions to the heat equation. For globally local graphs where $ f $ is given by a certain function of the volume growth we can therefore conclude stochastic completeness from our uniqueness class statement.

For a graph with an intrinsic metric that admits finite balls we then proceed as follows. It is known that adding edges that correspond to large jumps to a stochastically complete graph leaves the graph stochastically complete, see \cite{GHM12}. Hence, we can restrict ourselves to the case that the intrinsic metric has finite jump size. Together with the assumption of finite distance balls, finite jump size  gives that these graphs are now locally finite. Then, we {\em refine} the graph by adding vertices ``on edges'' such that the jump size of this new graph is  small enough far outside. For this modified graph the above mentioned uniqueness class theorem is applicable. As a final step we  show that stochastic completeness of  this refinement yields stochastic completeness of the original graph.


Let us be more specific.  	Let a graph $ b $ over $ (X,m) $ and a pseudo metric $ d
$ be given. The \emph{jump size} $ s $ of $ d $ is defined by
\begin{align*}
s=\sup\{d(x,y)\mid x,y \in X \text{ with } x\sim y\}.
\end{align*}
For the fixed reference vertex $ o\in X $ (which we mostly keep implicit in notation) and for $ r\ge0 $ we define  $ s_{r} $,  the \emph{jump size outside of} $ B_{r} $,  by
\begin{align*}
s_{r}:=\sup\{d(x,y)\mid x,y\in X\text{ with } x\sim y\mbox{ and }d(x,o)\wedge d(y,o)\ge r\}.
\end{align*}
Note that the jump size satisfies $s = s_0$. With these notions we can define what we mean by a graph being globally local.

\begin{definition}
	A graph with a pseudo metric  is called \emph{globally local} with respect to a  monotone increasing function $f:(0,\infty) \to (0,\infty)$ if $ s_{0} = s < \infty$ and there is a constant $A > 1$ such that
	\begin{align}\label{inequality:GL}
	 \limsup_{r \to \infty} \frac{s_r f(Ar)}{r}  <  \infty. \tag{GL}
	\end{align}
\end{definition}

The following remark puts this definition into perspective. 

\begin{remark}
By definition globally local graphs have finite jump size. Moreover,  a graph with finite jump size is globally local with respect to an increasing function $f$ iff there exist  $A > 1$ and $B> 0$ such that  
$$s_{r}\leq  \frac{Br}{f(Ar)}$$ 
holds for all  $r$ large enough. Therefore, globally local graphs are graphs of finite jump size for which $s_r$ shows a certain decay with respect to $f$, as $r \to \infty$. 

Putting the constant $A > 1$ into the definition of globally local graphs is a bit arbitrary. We chose this convention because then Theorem~\ref{theorem:uniqueness class} takes the same form as  the corresponding result on manifolds, see the discussion below.

Furthermore, note that $ r\mapsto r/f(Ar) $ does not need to be monotone decreasing. Since $ s_{r}\leq s <\infty $ for all  $ r\ge0 $, we immediately see that \eqref{inequality:GL} is only a restriction whenever $ f(r)>r $ for $ r $ large. Indeed, we often think of $ \lim_{r\to\infty} r/f(r)=0 $ as it is satisfied for example by $f(r)= r\log(r) $, $ r>0 $.
\end{remark}

The following result for globally local graphs is the main ingredient for the proof of the volume growth test for stochastic completeness. However, it is certainly of independent interest for the following reasons:
First of all it is the precise analogue of the corresponding theorem in the continuum due to Grigor'yan \cite{Grigoryan88}, which is the strongest known uniqueness class result on manifolds. Indeed, in the continuum the jump size is always zero, hence, there is no additional assumption on $ s_{r} $. Secondly, without assuming the graph to be globally local the result is  wrong even on the simplest graph, $ \Z $ with standard weights, as it was shown in an example in \cite{Hua11} (see also the remark after the theorem and Section \ref{section:examples}). Finally, our proof does not use the discreteness of the space $ X $ and hence extends to non-local operators on other state spaces.

\begin{theorem}[Uniqueness class]\label{theorem:uniqueness class}
Let $b$ be a graph over $(X,m)$ and let $d$ be an intrinsic pseudo metric  with finite distance balls. Assume that the graph is globally local with respect to a monotone increasing function $f:(0,\infty) \to (0,\infty)$ with 
	$$\int^\infty_{1} \frac{r}{f(r)} dr = \infty.$$
	Let $0< T \leq \infty$ and let $u:(0,T) \times X \to \R$ be a solution
	 to the heat equation with initial value $0$. If 
	\begin{align}\label{equation:growth condition 0}
	\int_0^T \sum_{x \in B_r} |u_t(x)|^2 m(x) dt \leq e^{f(r)} \tag{GC}
	\end{align}
	 for every $r > 0$, then $u_t = 0$ for all $t > 0$.
\end{theorem}

\begin{remark}
 Every graph with finite jump size is globally local with respect to the function $f:(0,\infty) \to (0,\infty),\, f(r) = r$. Hence, provided the intrinsic metric has finite distance balls and finite jump size, the above theorem can always be applied to solutions to the heat equation that show at most an exponential growth in the growth condition \eqref{equation:growth condition 0}.
 
 The best known   uniqueness class result that holds for all graphs that admit an intrinsic metric with finite distance balls and finite jump size is the following. For $\alpha > 0$ we let $f_\alpha:(0,\infty) \to (0,\infty),\, f_\alpha(r) = \alpha r \log r$.   It is shown in \cite{Hua11} that a solution to the heat equation with initial value $0$ vanishes if for some $\alpha < 1/2$ it satisfies the growth condition \eqref{equation:growth condition 0} with respect to the function $f_\alpha$.  Moreover, it is shown that this result is optimal in the following sense. For every $\varepsilon > 0$  there is a graph (indeed one can take the integer line $\Z$)  and a nontrivial solution to the heat equation with inital value $0$ such that for some $\alpha < 2 \sqrt{2} + \varepsilon$ it satisfies \eqref{equation:growth condition 0} with respect to the function $f_\alpha$.
 
 Our integral growth test for the function that controls the grwoth of the solution  cannot distinguish between the functions $f_\alpha$ for different $\alpha > 0$.  Thus, even if we could further weaken the assumption \eqref{inequality:GL}, the previous discussion shows that we cannot hope to recover the best uniqueness class result that holds for all graphs from our result on globally local graphs. However, for graphs that show the mild decay $s_r \leq K / \log r$ for some constant $K > 0$ and all $r$ large enough our theorem can be applied to solutions to the heat equation that satisfy the growth condition \eqref{equation:growth condition 0} with respect to $f_\alpha$ for some $\alpha > 0$. Hence, for such graphs our uniqueness class is stronger than the one obtained in \cite{Hua11}. A concrete example for this situation is discussed in Section~\ref{section:examples}. There we also show that our assumption \eqref{inequality:GL} is in some sense optimal for proving a uniquenss class with respect to a function $f$ that satisfies the above integral test. 
\end{remark}

Note that the assumptions of finite distance balls and finite jump size (which follows from the graph being globally local) already imply that the graph is locally finite, see e.g. \cite[Lemma~3.5]{Ke15}. Nevertheless, the  result about stochastic completeness, Theorem~\ref{thm:main}, which is derived from the previous theorem, does not assume local finiteness. The reason for this is a trick used in \cite{GHM12} that allows to remove edges corresponding to large jumps to obtain finite jump size and, hence, local finiteness. For the precise statement, see Lemma~\ref{lemma:truncation} below.

To apply the uniqueness class result to prove stochastic completeness for locally finite graphs, one needs the concept of refinements.  These are introduced next. Since the definition of a refinement is a bit technical  we explain the idea before hand. Given a graph $ b $ over $ (X,m) $ and a function $ n $ on the edges, we replace every edge $ \{x,y\} $ by a path $ \{x_{0},\ldots,x_{n(x,y)+1}\} $ of edges with $ x_{0}=x $ and $ x_{n(x,y)+1}=y $. Furthermore, the new metric can be thought as taking the distance in $ X $ and then adding the remaining part into the paths of edges that were added to the graph. Moreover, the edge weights and measure of the refinement are chosen in such a way that the volume growth of the refinement and the volume growth of the original graph can be compared.

 The definition below makes it precise on how to define the edge weights and the measure  given $ n $ and a metric $ d $, which will later assumed  to be intrinsic.
 %
 
\begin{definition}[Refinement]
Let $ b $ be a locally finite graph  over $ (X,m) $ and let $ d $ be a pseudo metric on $X$. Furthermore, let $ n:X\times X\to\N_{0} $ be symmetric and such that 
	$ n(x,y)\geq 1$ if and only if $ x\sim y $.
We  define the \emph{refinement  with respect to $ n $} to be the graph	$ b' $ over $ (X',m') $ with   distance function $ d' $ given as follows: The vertex set $ X' $ is the union
\begin{align*}
	X'=X\dot{\cup} {\bigcup_{x,y\in X}} X_{x,y},
\end{align*}
where $ X_{x,y} =X_{y,x}$ are pairwise disjoint finite sets (i.e., $ X_{x,y}\cap X_{w,z}=\emptyset $ if $ \{x,y\}\neq \{w,z\} $), such that $ \sharp X_{x,y}=n(x,y) $. Note that $ X_{x,y}=\emptyset $ if $x\not\sim y  $. The measure $ m' $ is defined as
\begin{align*}
m'\vert _{X}=m, \qquad m'(z) =\frac{2b(x,y)d(x,y)^{2}}{n(x,y) + 1},\;z \in X_{x,y} \text{ and }x\sim y.
\end{align*}
The edge weight $ b' $ is defined
as follows. For $ x\sim y $ in $ X $ we fix an enumeration of $ X_{x,y}=\{x_{1},\ldots, x_{n(x,y) }\} $  and define
\begin{align*}
	b'(x_{i-1},x_{i})&={b(x,y)}{(n(x,y) + 1) },\quad i=1,\ldots,n(x,y)+1,
\end{align*}
with $ x=x_{0} $, $ y=x_{n(x,y)+1} $. Otherwise, we set 
$$  b'=0  $$
and we denote $ x\sim' y $ if $ b'(x,y)>0 $. 
  
  We define $ d' $ in three steps. For $ z,z'\in\{x,y\}\cup X_{x,y} $,  we define
 \begin{align*}
 d'(z,z')=\min \{k\frac{d(x,y)}{n(x,y) + 1}\mid\mbox{$z= x_{0}\sim'\ldots\sim' x_{k}=z'$ in $\{x,y\}\cup X_{x,y}$}  \}.
 \end{align*}
For $ z\in X $, $ z'\in \{x,y\}\cup X_{x,y} $ with $ z\neq x,y $, we define
 \begin{align*}
 d'(z,z')=\min_{w\in\{x,y\}} \{d(w,z)+d'(w,z')  \}
 \end{align*}
 and for $ z\in X_{x,y} $  $ z'\in X_{x',y'} $ for $ \{x,y\}\neq \{x',y'\} $,  we define
 \begin{align*}
 d'(z,z')=\min_{v\in\{x,y\},v'\in\{x',y'\}}\{&{d'(z,v)}+d(v,v')+d'(v',z')\}.\end{align*}
\end{definition}

It is readily verified that $ d' $ is a pseudo metric and $ d=d' $ on $ X\times X$. 

Note that we define refinements only for locally finite graphs. This is because on locally infinite graphs the  edge weight of the refinement $b'$ might violate the summability condition $\sum_{y \in Y'} b'(x,y) < \infty$ for all $x \in X'$, that we always impose on edge weights. 

For  proving stochastic completeness of graphs through refinements the following stability result on stochastic completeness is essential. It is proven in Section~\ref{section:refinements}.

\begin{thm}[Stability of stochastic completeness under refinements]\label{thm:refinements}
	If a refinement of a locally finite graph with intrinsic metric is stochastically complete, then the graph is stochastically complete.
\end{thm}

\section{Grigor'yan's inequality and uniqueness class}

In this section we prove Grigor'yan's inequality, Lemma~\ref{lemma:grigoryans estimate}, for globally local graphs. It allows us to compare the size of a solution to the heat equation on a small ball at a given time with its size on a larger ball at an earlier time up to a small error. Here, the precise relation of the radii of the two balls with the possible time difference and the error is of utmost importance.  From this comparison we deduce Theorem~\ref{theorem:uniqueness class} on the uniqueness class for the heat equation by a standard iteration procedure. 

As in the manifold case Grigor'yan's inequality follows from carefully choosing the cut-off function $\varphi$ in the basic estimate of Lemma~\ref{lemma:general estimate}. However, since we deal with non-local operators, the estimates are more delicate. This manifests in the lengthy statement of Lemma~\ref{lemma:main estimate}, which can be seen as the main new technical ingredient of this paper. From it we deduce Grigor'yan's inequality for globally local graphs.

The following discrete Caccioppoli-type inequality is certainly well known to experts and is contained in the literature under various assumptions on the involved functions, see e.g. \cite[Lemma~1.6.1]{Hua11} and \cite[Lemma~3.4]{HKMW13}.

\begin{lemma}[Caccioppoli inequality] \label{lemma:Caccioppoli}
 Let $u \in \mathcal{F}(X)$ and let $\varphi \in C_c(X)$. Then
 $$-\sum_{x \in X} \mathcal{L}u(x) u(x) \varphi^2(x) m(x) \leq \frac{1}{2} \sum_{x \in X}u^2(x) |\nabla \varphi|^2(x). $$
\end{lemma}
\begin{proof}
 Under stronger assumptions on the geometry of the graph but less restrictions on $\varphi$ the inequality is contained in  \cite[Lemma~3.4]{HKMW13}. It is obtained by rearranging the  sums and an elementary estimate; the formal manipulations are given in the proof of \cite[Lemma~3.4]{HKMW13} and the lemmas preceding it. That the involved sums converge absolutely and so rearranging them is possible is guaranteed by the assumption $u \in \mathcal F (X)$ and $\varphi \in C_c(X)$. 
\end{proof}

 We use the Caccioppoli inequality to obtain the following a priori estimate. Implicitly it also appears in the proof of Grigor'yan's lemma in the continuum case, cf. \cite[Proof of Theorem~9.2]{Gri99}.

\begin{lemma}[Basic estimate]\label{lemma:general estimate}
Let $0 < T \leq \infty$ and let $u:(0,T) \times X \to \R$ be a solution to the heat equation with initial value $u_0$. Let $K \subseteq X$ finite. For $0 \leq t < T$, let $\varphi_t \in C(X)$ with $\mathrm{supp} \,\varphi_t \subseteq K$ be  such that for each $x \in X$ the function  $t \mapsto\varphi_{t}(x)$ is continuously differentiable on $(0,T)$ and continuous at $0$. Then for all $0< \delta \leq t$  we have
\begin{align*}
 \sum_{x \in X}u_t^2(x) &\varphi_t^2(x) m(x) \leq \sum_{x \in X}u_{t-\delta}^2(x)\varphi_{t-\delta}^2(x) m(x)\\  \
 &+\int_{t-\delta}^t \sum_{x \in X}u_\tau^2(x) \left(\partial_\tau\varphi_{\tau}^2(x) m(x) + |\nabla\varphi_{\tau}|^2(x)\right)  d\tau.
\end{align*}
\end{lemma}
\begin{proof}
 We set $s = t - \delta$ and evaluate the integral 
 $$I:= \int_{s}^t \sum_{x \in X} \partial_{\tau}u_\tau^2(x)\varphi_{\tau}^2(x)  m(x) d\tau $$
 in two ways. First we use partial integration and compute
 $$I = \left. \sum_{x \in X} u_\tau^2(x)\varphi_{\tau}^2(x)  m(x) \right|^t_s - \int_{s}^t \sum_{x \in X} u_\tau^2(x) \partial_{\tau}\varphi_{\tau}^2(x)  m(x) d\tau. $$
 Secondly, we use that $u$ solves the heat equation and the Caccioppoli type inequality of Lemma~\ref{lemma:Caccioppoli} to obtain
 \begin{align*}
  I &= 2\int_{s}^t \sum_{x \in X} \partial_{\tau}u_\tau(x)u_\tau(x)\varphi_{\tau}^2(x) m(x) d\tau\\ 
  &=- 2\int_{s}^t \sum_{x \in X} \mathcal{L}u_\tau(x)u_\tau(x)\varphi_{\tau}^2(x) m(x) d\tau \\
  & \leq \int_{s}^t \sum_{x \in X}u_\tau^2(x)  |\nabla\varphi_{\tau}|^2(x)  d\tau. 
 \end{align*}
 This finishes the proof.
\end{proof}

  The following lemma is the main new technical tool of this paper. It compares the size of a solution to the heat equation on a small ball at a given time with its size on a larger ball at an earlier time. More precisely, we show that  under a growth condition \eqref{inequality:estimate assumption}   a solution to the heat equation $u$  satisfies the inequality
$$\sum_{x \in B_r}  u_{t}^2(x)  m(x) \leq F \sum_{x \in B_{R}}u_{t-\delta}^2(x) m(x)  + \eqref{error1},$$
where \eqref{error1} is an  error term and $ F>1 $. An  error term of the form \eqref{error1} also appears in the manifold case, while  the constant $ F $  is a result of the non-local setting. For small enough $\delta$ (for a precise choice see Lemma~\ref{lemma:grigoryans estimate} below) and $R = E r$ for some constant $E > 1$, the error term \eqref{error1} is of order $r^{-2}e^{-f(Er)}$ if the graph is globally local with respect to $f$. This is Grigor'yan's estimate on globally local graphs, see Lemma~\ref{lemma:grigoryans estimate} below.

\begin{lemma}[Main technical estimate]\label{lemma:main estimate}
 Let $b$ be a graph over $(X,m)$ and let  $d$ be an intrinsic pseudo metric with finite distance balls and with finite jump size $s < \infty$.  Let $0 < T \leq \infty$ and let $u:(0,T) \times X \to \R$ be a  solution to the heat equation with initial value $u_0$. Let $f:(0,\infty) \to (0,\infty)$ be an increasing function such that for each $r > 0$ the following inequality holds
 \begin{align}
  \int_0^T  \sum_{x \in B_r} |u_t(x)|^2 m(x)dt\leq e^{f(r)}. \tag{GC} \label{inequality:estimate assumption}
 \end{align}
 For any $ 2s < r < R$, $0 < \lambda < 1$,  and $0 < \delta \leq t < T$, suppose $C>0, \varepsilon>0$ are constants such that 
 $$C\ge 8\exp\left(\frac{s_{r-2s }(4(R-r)+2s)}{C\delta}\right)\quad\mbox{ and }\quad\varepsilon\ge \frac{2s_{r-2s }^2}{C}. $$
 Then,
 \begin{align}
 \sum_{x \in B_r}  u_{t}^2(x)  m(x) \leq&  
 e^{\frac{\varepsilon}{C\delta}}\sum_{x \in B_{R}}u_{t-\delta}^2(x) m(x) \notag\\
&+  e^{\frac{2\varepsilon}{C\delta}}\frac{2}{(1-\lambda)^2(R - r)^2} e^{-\frac{ ( \lambda(R - r) - s)_+^2}{C \delta} + f(R + s)}. \tag{ET}\label{error1}
\end{align}
\end{lemma}
\begin{proof}
We set $r_\lambda = (1 - \lambda)r + \lambda R$. We apply Lemma~\ref{lemma:general estimate} to the functions $u$ and 

$$\varphi_\tau(x) = \eta(x) e^{\xi(x,\tau)},$$
where
$$\eta(x) = \left(1 - \frac{d(x,B_{r_\lambda})}{(R - r_\lambda)}\right)_+,$$
and  
$$\xi(x,\tau) = - \frac{\rho(x)^2 +\varepsilon}{C(t+\delta - \tau)},$$
with
$$\rho(x) = (d(x,o) - r)_+,$$ 
%
where $ x\in X $ and $ 0\leq \tau<t+\delta $.
The function $\varphi_{\tau}$ equals $\exp(- \frac{\varepsilon}{C(t+\delta - \tau)})$ on $B_r$, vanishes outside of $B_{R}$ and  satisfies $0 \leq \varphi_{\tau} \leq \exp(- \frac{\varepsilon}{C(t+\delta - \tau)})$. Hence, Lemma~\ref{lemma:general estimate} implies
\begin{align*} 
 e^{- \frac{2\varepsilon}{C\delta}}\sum_{x \in B_r} u_t^2(x)  &m(x) \leq e^{- \frac{\varepsilon}{C\delta}}\sum_{x \in B_{R}}u_{t-\delta}^2(x) m(x)\\
 &+\int_{t-\delta}^t \sum_{x \in X}u_\tau^2(x) \left(\partial_\tau\varphi_{\tau}^2(x) m(x) + |\nabla\varphi_{\tau}|^2(x)\right)  d\tau. 
 \end{align*}
In order to obtain the error term \eqref{error1}  we estimate
\begin{align*}
 \partial_\tau\varphi_{\tau}^2(x) m(x) + |\nabla \varphi_\tau|^2(x) =& 2\eta^2(x) e^{2\xi(x,\tau)} \partial_\tau \xi(x,\tau) m(x) \\
 &+ \sum_{y \in X}b(x,y) \left(\eta(x) e^{\xi(x,\tau)} - \eta(y) e^{\xi(y,\tau)} \right)^2. 
\end{align*}
Since 
$$\eta(x)e^{\xi(x,\tau)} - \eta(y)e^{\xi(y,\tau)} = \eta(x) (e^{\xi(x,\tau)} - e^{\xi(y,\tau)}) + e^{\xi(y,\tau)}( \eta(x)  - \eta(y)),$$
it satisfies
 \begin{align*} 
 \partial_\tau \varphi_{\tau}^2(&x) m(x) + |\nabla\varphi_{\tau}|^2(x)  \leq  2 \sum_{y \in X}b(x,y) e^{2\xi(y,\tau)}(\eta(x) - \eta(y))^2 \\
 &+ 2 \eta^2(x) e^{2\xi(x,\tau)} \left( \partial_\tau \xi(x,\tau) m(x) + \sum_{y \in X}b(x,y) (1- e^{\xi(y,\tau) - \xi(x,\tau)})^2\right).
 \end{align*}
 The terms which appear in the previous inequality are denoted by   
 $$I_1(x,\tau) = \sum_{y \in X}b(x,y) e^{2\xi(y,\tau)}(\eta(x) - \eta(y))^2$$
 and 
 $$I_2(x,\tau) = \partial_\tau \xi(x,\tau) m(x) + \sum_{y \in X}b(x,y) (1- e^{\xi(y,\tau) - \xi(x,\tau)})^2.$$
 Therefore, it suffices to show that the quantity
 \begin{align}\label{error term}
  2\int_{t-\delta}^t \sum_{x \in X} u_\tau^2(x) (I_1(x,\tau) + \varphi_\tau^2(x) I_2(x,\tau)) d\tau \tag{$\heartsuit$}
 \end{align}
 can be estimated by the error term \eqref{error1}.

 We start with controlling the summand containing $I_1$ by \eqref{error1}. The function $\eta$ equals $1$ on $B_{r_\lambda}$ and equals $0$ on $X\setminus B_{R}$. Hence, $I_1(x,\tau)$ vanishes for $x \in B_{r_\lambda - s} \cup  X \setminus B_{R+s}$. Also for $x \in B_{R+s} \setminus B_{r_\lambda - s}$ and $ y\in B_{r_\lambda - s} $ with $ y\sim x $ the term $ (\eta(x)-\eta(y)) $ vanishes. Moreover, for   $y\in X\setminus  B_{r_{\lambda}-s} $  we obtain
 $$- \rho(y)^2 \leq - (r_\lambda - r - s)_+^2 = - ( \lambda(R - r) - s)_+^2. $$   
 If, additionally, $t-\delta \leq \tau \leq t$, this implies
 $$e^{2\xi(y,\tau)} \leq e^{-\frac{2 ( \lambda(R - r) - s)_+^2+2\varepsilon}{C (t + \delta - \tau)}} \leq e^{-\frac{ ( \lambda(R - r) - s)_+^2+\varepsilon}{C \delta}}.$$
 With  these observations an integration over space and time yields 
\begin{align*} 
  &\int_{t-\delta}^t \sum_{x \in X}u_\tau^2(x) I_1(x,\tau)\, {\rm d}\tau\\
  &= \int_{t-\delta}^t \sum_{x \in B_{R+s}\setminus B_{r_\lambda-s}}u_{\tau}^2(x)\sum_{y \in X}b(x,y) e^{2\xi(y,\tau)}(\eta(x) - \eta(y))^2\, {\rm d}\tau \\ 
  &\leq \int_{t-\delta}^t \sum_{x \in B_{R+s}\setminus B_{r_\lambda-s}}u_{\tau}^2(x)  e^{-\frac{ ( \lambda(R - r) - s)_+^2+\varepsilon}{C \delta}} \sum_{y \in X\setminus  B_{r_{\lambda}-s}}b(x,y) (\eta(x) - \eta(y))^2\, {\rm d}\tau \\ 
  &\leq  \frac{1}{(1-\lambda)^2(R - r)^2} e^{-\frac{ ( \lambda(R - r) - s)_+^2+\varepsilon}{C \delta}}  \int_{t-\delta}^t \sum_{x \in B_{R+s}\setminus B_{r_\lambda-s}}u_{\tau}^2(x) m(x)\, {\rm d}\tau \\ 
  &\leq    \frac{1}{(1-\lambda)^2(R - r)^2} e^{-\frac{ ( \lambda(R - r) - s)_+^2+\varepsilon}{C \delta} + f(R + s)}.
 \end{align*}
 For the second to last inequality we used that $\eta$ is $R - r_\lambda = (1-\lambda)(R-r)$-Lipschitz and that the metric $d$ is intrinsic. For the last inequality we used that $u$ satisfies the growth condition \eqref{inequality:estimate assumption}. This yields the error term \eqref{error1} in the desired inequality.
 
 We now estimate $I_2(x,\tau)\leq 0$ whenever it appears in  \eqref{error term}. As seen above in the discussion preceding \eqref{error term}, we need to multiply $I_2$ by $u^2\varphi^2$ and then integrate over space and time. Since $\varphi$ vanishes on $X \setminus B_{R}$, it suffices to control $I_2$ on $B_{R}$. Obviously we have 
 $$\partial_\tau \xi(x,\tau) m(x) = - \frac{\rho^2(x)+\varepsilon}{C(t+\delta - \tau)^2}m(x) \leq 0.$$
   Whenever $x \in B_{r-s}$ and $y \sim x$ the function $\rho$ and the exponent $\xi$ satisfy $\rho(x) = \rho(y) = 0 = \xi(x,\tau) = \xi(y,\tau)$. Therefore, $I_2\le 0$ also holds on $B_{r-s}$.
   
    We employ the following observation and some elementary inequalities to estimate $I_2$ on $B_{R}\setminus B_{r-s}$. For $t - \delta \leq \tau \leq t$, $x \in B_{R}\setminus B_{r-s}$ and $x \sim y$ we have
 \begin{align*}
  |\xi(x,\tau) - \xi(y,\tau)| &= \frac{|\rho(x) - \rho(y)||\rho(x) + \rho(y)|}{C(t+\delta - \tau)} \notag \\
  &\leq \frac{d(x,y)(2\rho(x) + d(x,y))}{C(t + \delta - \tau)} \notag \\
  &\leq \frac{ s_{r-2s} (2(R-r) + s)}{C\delta}.
 \end{align*}
 Note that for the last inequality we used  $d(x,y) \leq s_{r-2s}$ whenever $x \in X \setminus B_{r-s}$ and $y \sim x$ as well as $ \rho(x)\leq R-r $, $ x\in B_{R} $ and $ d(x,y)\leq s $, $ x\sim y $. 
 As we proceed, we use these estimates and  the following elementary inequalities
 \begin{equation*}
(1-e^a)^2\le 
a^2 e^{2a\vee 0} \qquad \mbox{and}\qquad (a+b)^2\le 2(a^2+b^2),\qquad a,b\in \R.
 \end{equation*}
  Let $t - \delta \leq \tau \leq t$ and $x \in B_{R}\setminus B_{r-s}$. Then, for $ y\sim x $ with $ \rh(y)\leq \rh(x) $
  \begin{align*}
  (1- e^{\xi(y,\tau) - \xi(x,\tau)})^2 &\leq |\xi(x,\tau) - \xi(y,\tau)|^2 e^{2|\xi(y,\tau) - \xi(x,\tau)|}\\
  &\leq \frac{d(x,y)^2(\rho(x) + \rho(y))^2}{C^2(t + \delta - \tau)^2} e^{  \frac{ s_{r-2s} (4(R-r) + 2s)}{C\delta}}\\
  &\leq \frac{4\rho^2(x)}{C^2(t + \delta - \tau)^2}e^{  \frac{ s_{r-2s} (4(R-r) + 2s)}{C\delta}} d(x,y)^{2}
  \end{align*}
  and, for $ y\sim x $ with $ \rh(y)> \rh(x) $,
  \begin{align*}
   (1- e^{\xi(y,\tau) - \xi(x,\tau)})^2 &\leq |\xi(x,\tau) - \xi(y,\tau)|^2\\
   &\leq\frac{d(x,y)^2(\rho(x) + d(x,y))^2}{C^2(t + \delta - \tau)^2}\\
    &\leq\frac{4\rho^{2}(x) + 2s^2_{r-2s}}{C^2(t + \delta - \tau)^2}d(x,y)^2
  \end{align*}
  where we use in the last inequality that $ d(x,y)\leq s_{r-2s} $ since $ x\in X\setminus B_{r-s} $ and $ y\sim x $.
  Thus, taking these two estimates together we obtain
 \begin{align*}
  \sum_{y \in X}&b(x,y) (1- e^{\xi(y,\tau) - \xi(x,\tau)})^2 \\
  &\leq   \frac{4\rho^2(x)e^{  \frac{ s_{r-2s} (4(R-r) + 2s)}{C\delta}}+4\rho^{2}(x) + 2s^2_{r-2s}}{C^2(t + \delta - \tau)^2} \sum_{y\in X}b(x,y)d(x,y)^{2}\\
   &\leq \frac{8\rho^2(x)e^{  \frac{ s_{r-2s} (4(R-r) + 2s)}{C\delta}}+2  s_{r-2s}^2}{C^2(t + \delta - \tau)^2}m(x) .
\end{align*}
 Altogether, for $t-\delta \leq \tau \leq t$ we obtain $I_2(x,\tau) \le 0$ whenever $x \in B_{r-s}$ and,  for $x\in B_R\setminus B_{r-s}$, we obtain
\begin{align*} 
 I_2(x,\tau) &= -\frac{\rho^2(x)+\varepsilon}{C(t+\delta - \tau)^2} m(x) + \sum_{y \in X}b(x,y) (1- e^{\xi(y,\tau) - \xi(x,\tau)})^2\\
 &\leq \frac{ m(x)}{C(t+\delta-\tau)^2} \left(\rho^2(x)\left(\frac{8}{C}e^{  \frac{ s_{r-2s} (4(R-r) + 2s)}{C\delta}} -1 \right)+  \left(\frac{2s_{r-2s}^2}{C}-\varepsilon\right)\right)\\
 &\leq 0
\end{align*}
by our assumptions on $C$ and $\varepsilon$. This finishes the proof.
\end{proof}


Next, we choose the parameters in the main technical estimate above to obtain a discrete version of Grigor'yan's inequality, \cite[Theorem 9.2]{Gri99}.

\begin{lemma}[Grigor'yan's inequality]\label{lemma:grigoryans estimate}
 Assume the situation of Lemma~\ref{lemma:main estimate}. Assume further that  the underlying graph with the given metric is globally local with respect to $f$.  Then there exist constants $r_0 ,D,E,F,G  > 1$  such that for all  $r \geq r_0$ and  $0 <\de \leq t < T$  with
 $$\delta  \le \frac{r^2}{D f(Er)}$$
 the following inequality holds
 \begin{align}\label{inequality:grigoryans estimate}
  \sum_{x \in B_r}  u_{t}^2(x)  m(x) \leq F\sum_{x \in B_{Er}}u_{t-\delta}^2(x) m(x) +  \frac{Ge^{-f(Er)}}{r^2}. \tag{GI}
 \end{align}

\end{lemma}
\begin{proof}
%
%
From the globally local assumption we have $ s_{0} =  s < \infty$ and that there exists $ A>1 $ and $ B>0 $ such that $ s_{r}\leq Br/f(Ar) $ for  all $r$ large enough.

Let $ 1<E'<E<A $ and $ R=E'r $. In order to determine the parameter $ r_{0} $ such that the asserted estimate holds we do not keep track of the precise constant but rather speak of estimates holding for $ r $ large enough. Whenever we do so, ``large enough'' will only depend on the constants $ E',E,A $ as well as on the jump size $ s $.

For a parameter $ \al $,  which is to be  chosen later, and variable $C >0$,  we
consider
\begin{align*}
\de(r,\al,C)= \frac{r^{2}}{\al Cf(Er)}
\qquad\mbox{
and}
\qquad
\varepsilon(r,C)=\frac{2s_{r-2s}^{2}}{C}.
\end{align*}
In order to apply the main technical estimate, Lemma~\ref{lemma:main estimate}, the constant $ C $ has to satisfy a lower bound depending on $\delta(r,\al, C), r$ and $R = E'r$. Using that the graph is globally local, we show that for our choice of parameters this bound is satisfied for $C$ large enough depending only on $A,B$ and $\alpha$.
To this end  consider the function 
 \begin{align*}
a(r):=\frac{\al s_{r-2s
	}(4(E'-1)r+2s)f(Er)}{r^{2}} =\frac{s_{r-2s	}(4(R-r)+2s)}{C\de(r,\al,C)},
\end{align*}	
where the last equality follows directly from the definition of $ \de $ and $ R=E'r $.
Note that for sufficiently large $ r $ we have $ f(Er)\leq f(A(r-2s)) $ since $ f $ is monotone increasing and $1 < E < A$. So, by the assumption $ s_{r}\leq Br/f(Ar) $  we infer that $ a $ can be bounded via the estimate
\begin{align*}
a(r)\leq 4\al BE'\frac{(r-2s)(r+2s)f(Er)}{r^{2}f(A(r-2s))}\leq 4\al BE'\leq 4\al A B
\end{align*}
for sufficiently large $ r $. Hence, choosing $ C\ge 8 e^{4\al A B} $ gives immediately
\begin{align*}
C\ge 8e^{4\al A B} \ge 8e^{a(r)}=8\exp\left({\frac{s_{r-2s	}(4(R-r)+2s)}{C\de(r,\al,C)}}\right)
\end{align*}
for $ r $ large enough. With this choice of $ C $, $ \varepsilon(r,C) $, $ \de(r,\al,C) $, we get with  $ R=E'r $, $ \lambda=1/2 $ as well as $ T>t\geq \de(r,\al,C)>0 $
 \begin{align*}
\sum_{x \in B_r}  u_{t}^2(x)  m(x) \leq&  
e^{\frac{\eps(r,C)}{C\de(r,\al,C)}}\sum_{x \in B_{E'r}}u_{t-\delta(r,\al,C)}^2(x) m(x) \notag\\
&+  e^{\frac{2\eps(r,C)}{C\de(r,\al,C)}}\frac{8}{(E' -1)^{2} r^2} e^{-\al \frac{ ((E'-1)r/2 - s)_+^2}{ r^{2}}f(Er) + f(E'r + s)}
\end{align*} 
for all $ r $ large enough by the main technical estimate, Lemma~\ref{lemma:main estimate}, above.

We proceed by estimating the exponential factor at the end of the second term by choosing $ \al $. There is $ \al>0 $ (which only depends on $ E' $ and $ s $) such that for all  $ r $ large enough  we have
\begin{align*}
\al \frac{ ((E'-1)r/2 - s)_+^2}{ r^{2} }\ge 2.
\end{align*}
Since  $Er\ge E'r+s  $ for $ r $ large enough and since $f$ is increasing, for large enough $r$ this  implies 
\begin{align*}
e^{-\al \frac{ ((E'-1)r/2 - s)_+^2}{ r^{2}}f(Er) + f(E'r + s)} \leq e^{-f(Er)}.
\end{align*}

Next, we bound the term $ \exp(\varepsilon/C\de) $. We use the bounds $ s_{r}\le Br/f(Ar) $, $ A>E>1 $,   $ C \ge  8 e^{4\al A B}$, to estimate
\begin{align*}
\frac{\varepsilon(r,C)}{C\delta(r,\al,C)}=\frac{2\al s_{r-2s}^{2}f(Er)}{Cr^{2}}\leq 2\al B^2 \frac{({r-2s})^{2}f(Er)}{Cr^{2}f(A(r-2s))^{2}}\leq 
\frac{\alpha B^2e^{-4\alpha AB}}{4  f(A(r-2s))}
\end{align*}
for $ r $ large enough. Since $f$ is increasing, there is  a constant $ F>1 $ such that 
\begin{align*}
e^{\frac{\varepsilon(r,C)}{C\delta(r,\al,C)}}\leq F
\end{align*}
for all  $ C\ge  8 e^{4\al AB} $ and $ r $ large enough.

As a last step, we choose 
\begin{align*}
G=\frac{8F^2}{(E'-1)^2}\quad \mbox{ and }\quad D=  8\al e^{4\al AB} .
\end{align*}
Note, that  the constants $\al,D, E,E',F,G$ do not depend on $ C $.  Since any $0< \de \leq t <T$ with
\begin{align*}
\de  \leq \frac{r^{2}}{Df(Er)}
\end{align*}
can be written as
$$\delta = \delta(r,\alpha,C) = \frac{r^{2}}{\alpha C f(Er)} $$
with $C \geq   8 e^{4\al AB}$, the above considerations show
 \begin{align*}
\sum_{x \in B_r}  u_{t}^2(x)  m(x) \leq&  
F\sum_{x \in B_{E'r}}u_{t-\delta}^2(x) m(x) +  G\frac{e^{-f(Er)}}{ r^2}
\end{align*} 
 for such $\delta$ and all $r$ large enough. This finishes the proof.
 \end{proof}

With the help of Grigor'yan's inequality we can establish the uniqueness class. The proof is based on an iteration procedure that is basically the same as on manifolds, cf. the proof of \cite[Theorem~9.2]{Gri99}.  Below we give the details for the convenience of the reader.

\begin{proof}[Proof of Theorem~\ref{theorem:uniqueness class}]  Let $u:(0,T) \times X \to \R $ be a solution to the heat equation with initial value $0$ as in the theorem. Moreover, let $r_0 > 0$ and $D,E, F,G > 1$  be  constants as in  Lemma~\ref{lemma:grigoryans estimate}.

We can assume without loss of generality that  $ f(r)\ge r $: Otherwise replace $ f $ with $ g(r)=f(r)\vee r $, which obviously satisfies $ e^{f(r)}\leq e^{g(r)} $. Moreover, 
$$  \int_{1}^{\infty}\frac{r}{g(r)} dr=\mathrm{Leb}(M)+\int_{[1,\infty)\setminus M}\frac{r}{f(r)}dr, $$ 
with $ M=\{r\ge 1\mid f(r)< r\} $. If $ \mathrm{Leb}(M)=\infty $, then $ \int r/g(r)dr=\infty $. Otherwise the bounded function $M \to  \R, r\mapsto r/f(r) $ is integrable over $ M $ and, therefore, the second integral over $ [1,\infty)\setminus M $ is infinite
whenever $ \int_{1}^{\infty}r/f(r) dr  =\infty$. Finally, using $ s_{r}\leq s_{0} =   s < \infty $ the globally local assumption $s_{r}\leq B'r/ g(Ar) $ for large $r$ is satisfied  with the slight modification of replacing the original $ B $ by $B' =  B\vee As $.

Thus,  there exists a constant $ H $ such that $$ F^{k}e^{-f(R_{k+1})}=  F^{k}e^{-f(E^{k+1}r)}\leq F^{k}e^{-E^{k+1}r}\leq H  $$ for all $ r\ge r_{0} $ and $ k\in\N_{0} $.


Let $r \geq r_0$  and   $R_k := E^k r$ for $k \geq 0$. Since $\int_1^\infty \frac{r}{f(r)}dr = \infty$ and $f$ is monotone increasing, it follows that
$$\sum_{k = 1}^\infty \frac{R_k^2}{f(ER_k)} = \infty. $$
Hence, for every $0 < t < T$ there exists a natural number $N$ and 
$$ \delta_k \leq \frac{R_k^2}{D f(ER_k)},\quad  k = 0,1,\ldots,N,$$
with $\sum_{k = 0}^N \delta_k = t$. An iterative application of Grigor'yan's inequality \eqref{inequality:grigoryans estimate} with the radii $R_k=E^{k}r$ and the time differences $\delta_k$ yields 
\begin{align*}
\sum_{x \in B_r} u_t(x)^2 m(x) &\leq F\sum_{x \in B_{R_1}} u_{t - \delta_0}^2(x) m(x) + \frac{Ge^{-f(R_{1})}}{R_0^2}\\
&\leq F^{N+1} \sum_{x \in B_{R_{N+1}}} u^2_{t - \sum_{k=0}^{N}\de_{k}}(x) m(x) + G \sum_{k = 0}^N \frac{F^{k}e^{-f(R_{k+1})}}{R_k^{2}}\\
&\leq  \frac{GHE^{2}}{E^{2}-1}\frac{1}{r ^{2}}.
\end{align*}
For the last inequality we used $ \sum_{k=0}^{N}\de_{k} = t$ and $ u_0 = 0$. Now, letting $r \to \infty$ yields $u_t = 0$ and the theorem is proven. 
\end{proof}

\section{Refinements}\label{section:refinements}

In this section we study properties of refinements of locally finite graphs.  We prove the stability of stochastic completeness under refinements and compare the volume growth of refinements with that of the given graph. Moreover, given an  increasing function $f:(0,\infty) \to (0,\infty)$ we show that every locally finite graph has a refinement that is globally local with respect to $f$. At the end of this section we use these insights to prove our main result Theorem~\ref{thm:main}. As a tool to reduce the proof to the case of locally finite graphs with finite jump size we also discuss the stability of stochastic completeness under truncating the edge weights.

Let $ b $ be a graph over $ (X,m) $, let $ d $ be a pseudo metric on $X$ and let $ n:X\times X\to\N_{0} $ be a symmetric function such that $ n(x,y)=0 $ if and only if $ b(x,y)=0 $ for $ x,y\in X $. For the refinement $ b' $ over $ (X',m') $ with respect to $ n $, we denote by $ \mathcal{L}' $ the Laplacian of $ b' $.

We prove  Theorem~\ref{thm:refinements} via  a slight modification of the weak Omori-Yau principle, which is taken from \cite[Theorem~2.2]{Huang11}.

\begin{lemma}[Weak Omori-Yau principle]\label{lemma:omori yau}
	Let $ b $ be a graph over $ (X,m) $.   The graph is stochastically incomplete if and only if there is $c > 0$ and $ \alpha>0 $ and a function   $ u\in \mathcal{F}(X) $ with $ \sup u < \infty$ such that $ \mathcal{L}u < -c$ on $ \Omega_{\al}=\{x\in X\mid u(x) > \sup u-\al\} $.
\end{lemma}
\begin{proof}
 In \cite[Theorem~2.2]{Huang11}  the function $u$ in the statement is nonnegative. However, the proof given there works for bounded above $u \in \mathcal F (X)$.
\end{proof}

Next, we prove the stability result, Theorem~\ref{thm:refinements}, which states that if a refinement is stochastically complete, so is the original graph.

\begin{proof}[Proof of Theorem~\ref{thm:refinements}]
	Let $ b $ be a locally finite graph over $ (X,m) $ with intrinsic metric $d$ and let $n:X \times X \to \N_0$ with $n(x,y) \geq 1$ if and only if $x \sim y$. Let $ b' $ be the refinement over $ (X',m') $ with respect to  $ n $.
	
	Suppose that the given graph $b$ over $(X,m)$ is stochastically incomplete. After scaling, the weak Omori-Yau principle yields a bounded above   function  $ u \in \mathcal F (X) $ and $\alpha > 0$ such that $ \mathcal{L}u < -1 $ on the set $ \Omega_{\alpha}=\{x\in X\mid u(x)> \sup u - \al\} $. For $ x,y\in X $ with $ x\sim y $ let
	\begin{align*}
	\ph_{x,y}:[0,d(x,y)]\to \R, \quad 
	t\mapsto \frac{1}{2}t^{2}+\left(\frac{{u(y)-u(x)}}{d(x,y)}-\frac{d(x,y)}{2}\right)t+u(x).
	\end{align*}
	We write $\{x,y\} \cup X_{x,y}=\{x_{0},\ldots,x_{n(x,y)+1}\} $ with $ x=x_{0}\sim'\ldots\sim' x_{n(x,y)+1}=y $ and we define
	\begin{align*}
	u'(x_{i})=\ph_{x,y}\left(i \frac{d(x,y)}{n(x,y)+1}\right),\quad i=0,\ldots,n(x,y) + 1.
	\end{align*}
	Such an enumeration is unique up to reversing the order. However, since $ \ph_{x,y}(t)=\ph_{y,x}(d(x,y)-t) $, the function $u'$ is well-defined, i.e., it is independent of the enumeration of $X_{x,y}$. 
	%
	Moreover, it satisfies
	$$ u'\vert_{X}=u .$$ 
	We prove that $u'$ is bounded from above  and satisfies $\mathcal{L}'u' < -1/2$ on 
	$$\Omega_{\al}':=\{x'\in X'\mid u'(x') > \sup u' - \alpha\}.$$

	Since the second derivative of $\varphi_{x,y}$ equals $1$, the function $\varphi_{x,y}$ is convex and so $ \ph_{x,y}\leq u(x)\vee u(y)$. Therefore, $u'$ is bounded from above by $\sup u$ and
	\begin{align*}
	\Omega_{\al}'\subseteq \Omega_{\alpha}\cup\bigcup_{x\in \Omega_{\alpha}}\bigcup_{y\sim x}X_{x,y}.
	\end{align*}
	For $ x\in \Omega_{\alpha} $ and $y' \in X'\setminus X$ with $x \sim' y'$ there exists a unique $y \in X$ with $x \sim y$ and $y' \in X_{x,y}$. In this case, $b'(x,y') =b(x,y)(n(x,y) + 1)$, $m'(x) = m(x)$, $u(x) = u'(x)$ and $u'(y') = \ph_{x,y}(d(x,y)/(n(x,y)+1))$. Hence, with $n := n(x,y)$ we obtain
	\begin{align*}
	&b'(x,y')(u'(x) - u'(y')) \\
	&= b(x,y)(n + 1) \left(- \frac{d(x,y)^2}{2(n+1)^2} + \frac{(u(x) - u(y))}{n + 1} + \frac{d(x,y)^2}{2(n + 1)}\right) \\
	&\leq b(x,y)(u(x) - u(y)) + \frac{b(x,y)d(x,y)^2}{2}.
	\end{align*}
		For $ x\in \Omega_{\alpha} $ and $y' \in X$ with $ x\sim' y' $ the inequality above holds trivially.
	
	Summing up these inequalities over $y'$, and using that $d$ is intrinsic and $\mathcal{L} u < -1$ on $\Omega_\alpha$ yields
	\begin{align*}
	\mathcal{L}'u'(x)&=\frac{1}{m'(x)}\sum_{y'\in X'}b'(x,y')(u'(x)-u(y'))\\
	&\leq \frac{1}{m(x)}\sum_{y\in X}b(x,y)(u(x) - u(y)) + \frac{1}{2m(x)}\sum_{y\in X}b(x,y)d(x,y)^2 \\
	&< -1 + \frac{1}{2} = - \frac{1}{2}.
	\end{align*}
	Moreover, for  $ x\in \Omega_{\alpha} $ and $y \in X$ with $x \sim y$ and $ n:=n(x,y)\ge 1 $ let $ X_{x,y}\cup\{x,y\} =\{x_{0},\ldots,x_{n+1}\}$ such that $x= x_{0}\sim'\ldots\sim'x_{n+1}=y$. We set $d:= d(x,y)$. For $ i= 1,\ldots,n $  a straightforward computation shows
	\begin{align*}
       	\mathcal{L}'u'(x_{i})&=\frac{1}{m'(x_{i})}\sum_{e\in\{\pm 1\}} b'(x_{i},x_{i+e})\left(\ph_{x,y}\left(\frac{id}{n+1}\right)-\ph_{x,y}\left(\frac{(i + e)d}{n+1}\right)\right) \\
       	&=-\frac{1}{2}.
	\end{align*}
	Hence, by the weak Omori-Yau principle we infer stochastic incompleteness of the refinement.
\end{proof}

We denote the distance balls with respect to $ d $ and $ d' $ about a fixed reference point $ o \in X$ by $ B_{r} $ respectively $ B'_{r} $.
\begin{lemma}\label{lemma:properties of refinements}
	Let $ b $ be a locally finite graph over $ (X,m) $ and $ d $ be an intrinsic pseudo metric.  For the refinement $ b' $ over $ (X',m') $ with respect to a function $ n $ the following holds:

	\begin{itemize}
		\item [(a)]  $ d '$ is an intrinsic pseudo metric such that $ d=d' $ on $ X\times X $.
		\item [(b)] $ m(B_{r})\leq m'(B_{r}')\leq 2m (B_{r}) $  for all $ r\ge0 $.
		\item [(c)] The ball $ B_{r}' $ is finite if and only if $ B_{r} $ is finite  for all $ r\ge0 $.
	\end{itemize}
\end{lemma}
\begin{proof} (a): That $d'$ is a pseudo metric with $ d=d' $ on $ X\times X $ readily follows from the definition of $ d' $ as already remarked above. We show that it is intrinsic. Since $d$ is intrinsic, for  $ x\in X $ we have
	\begin{align*}
	\sum_{y'\in X'}b'(x,y')d'(x,y')^{2}&=	\sum_{y\in X,y'\in X_{x,y}}b'(x,y')d'(x,y')^{2}\\
	&=\sum_{y\in X} b(x,y)(n(x,y)+1)\frac{d(x,y)^{2}}{(n(x,y)+1)^{2}}\\
	&\leq m(x)=m'(x).
	\end{align*}
  Furthermore,  for $ x\sim y $ every 	 $ x'\in X_{x,y} $ has only two neighbors $ y',y'' $. Moreover, $ b(x',y')=b(x',y'')=b(x,y)(n(x,y) + 1)$ and  $d'(x',y')=d'(x',y'')=d(x,y)/(n(x,y) + 1)  $. Therefore,
		\begin{align*}
			\sum_{z\in X'}b'(x',z) d'(x',z)^{2}=\frac{2b(x,y)d(x,y)^{2}}{n(x,y) + 1}=m'(x')
		\end{align*}
		by definition of $ m' $. 
	
	(b): By (a) we have $ B_{r}=B_{r}'\cap X $. Hence, $$ B_{r}\subseteq B_{r}' \subseteq B_{r}\cup \bigcup_{x\in B_{r},y\in X} X_{x,y}, $$
	where $ X_{x,y}=\emptyset $ if $ x\not\sim y $ with respect to $ b $.
	 Since $ m'\vert_{X}=m $, we have
	\begin{align*}
	m(B_{r})\leq m'(B_{r}').
	\end{align*}
	Furthermore, for given $ x \in X $ we have by the intrinsic metric property of $d  $
	\begin{align*}
	m'\left (\bigcup_{y\in X}X_{x,y}\right )=\sum_{y\sim x} n(x,y) \frac{b(x,y)d(x,y)^{2}}{n(x,y) + 1}\leq m(x).
	\end{align*}
	Hence, since $ m=m' $ on $ X $ we get by the estimate above
	\begin{align*}
	 m'(B_{r}')\leq m(B_{r})+ \sum_{x\in B_{r}}m'\left(\bigcup_{y \in X}X_{x,y}\right)\leq 2m(B_{r}).
	\end{align*}
	This proves statement (b).
	
	(c): This readily follows from the inclusion in the beginning of the proof of (b).
\end{proof}

Next, we show that for every function bounding the jump size of a graph outside of balls there exists a function $ n $ such that the corresponding refinement satisfies this bound. To this end we say that a function $ g:(0,\infty)\to(0,\infty) $ is {\em uniformly positive} on a set $ M\subseteq (0,\infty) $ if there is $ C_{M}>0 $ such that $g\ge C_{M}  $ on $ M $. Note that if $ f:(0,\infty)\to(0,\infty) $ is monotone increasing, then $ r\mapsto f(r) $ is uniformly positive on every compact set. In the following lemma we denote by $s'_r$ the jump size outside of balls in the refinement, i.e., 
$$s_{r}':=\sup\{d'(x',y')\mid  x',y'\in X'\text{ with } x'\sim' y'\mbox{ and }d'(x',o)\wedge d'(y',o)\ge r\}.$$

\begin{lemma}\label{l:refinement}
	Let $ b $ be a locally finite graph over $ (X,m) $ and let $d$ be a metric on $X$. For every  function $ g:(0,\infty)\to (0,\infty) $, which is uniformly positive on every compact set, there exists a symmetric $ n :X \times X \to \mathbb N_0 $ with $n(x,y) \geq 1$ if and only if $x \sim y$,  such that the refinement $ b' $ over $ (X',m') $ with metric $ d' $  with respect to $n$ satisfies 
	\begin{align*}
	s_{r}' \leq g(r),\qquad r\geq 1.
	\end{align*}
\end{lemma}
\begin{proof}
Let $x,y \in X$ with $x \sim y$ and let $r_{x,y} := \max\{d(x,o),d(y,o)\} + d(x,y)$. For each such pair of vertices, we choose $n(x,y) \in \N$ so large  that 
$$\frac{d(x,y)}{\inf_{r\in[1,r_{x,y}]}g(r_{x,y})} \leq n(x,y) + 1,$$
which is possible since $ \inf_{r\in[1,r_{x,y}]}g(r)>0 $. If $x \not  \sim y,$ we let $ n(x,y) = 0$.

 We  prove that the refinement with respect to $n$ has the desired properties.  Let $ r\ge1 $, $x',y' \in X'$ with $x' \sim' y'$ and $d'(x',o)\wedge d'(y',o)\ge r$. By the definition of the refinement there exist unique $x,y \in X$ with $x \sim y$ such that $x',y' \in \{x,y\} \cup X_{x,y}$. 
 Then, 
 $$r \leq d'(x',o) \leq d(x,o) + d'(x',x) \leq d(x,o) + d(x,y) \leq r_{x,y}.$$ 
 Therefore, by the definition of $d'$ and $n$ we have 
 $$d'(x',y') = \frac{d(x,y)}{n(x,y) + 1} \leq \inf_{r'\in[1,r_{x,y}]}g(r')\leq g(r).$$
Hence, we obtain $ s'_{r}\leq g(r) $ for $ r\ge 1 $.
\end{proof}

Refinements are only defined for locally finite graphs. Hence, the stability of stochastic completeness under refinements can only be employed for locally finite graphs. The way of incorporating locally infinite graphs into our main theorem is through  the stability of stochastic completeness under adding large jumps and using that graphs with finite distance balls and finite jump size are locally finite. This is discussed next.

Given a graph $b$ over $(X,m)$ with intrinsic pseudo metric $d$ and $0 < s < \infty$ we define the {\em truncated edge weight} $b_s := b 1_{\{d \leq s\}}$. The graph $b_s$ over $(X,m)$ with pseudo metric $d$ has jump size $s$. Moreover, $d$ is  intrinsic with respect $b_s$ over $(X,m)$. 

\begin{lemma} \label{lemma:truncation}
 Let $b$ be a graph over $(X,m)$ with intrinsic pseudo metric $d$ and let $s> 0$. If $b_s$ over $(X,m)$ is stochastically complete, then $b$ over $(X,m)$ is stochastically complete.
\end{lemma}
\begin{proof}
 We  show that if $b$ is stochastically incomplete, then $b_s$ is stochastically incomplete. We denote by $\mathcal L_s$ the formal Laplacian of the graph $b_s$ over $(X,m)$.

 Assume that $b$ is stochastically incomplete. After scaling the weak Omori-Yau principle, Lemma~\ref{lemma:omori yau}, yields a bounded above function $u \in \mathcal F(X)$ and $\alpha > 0$ with $\mathcal L u < -1$ on $\Omega_\alpha = \{x \in X \mid u(x) >  \sup u - \alpha\}$. By adding a constant function we can also assume $0 < \sup u < s^2$.  Moreover, since $\Omega_\alpha \subseteq \Omega_\beta$ if $\alpha < \beta$, we can further assume $\alpha < \sup u$ so that $u > 0$ on $\Omega_\alpha$. Using that $d$ is intrinsic, for $x \in \Omega_\alpha$ we obtain
 \begin{align*}
  \mathcal L_s u(x) &= \mathcal L u(x) - \frac{1}{m(x)}\sum_{y \in X, d(x,y) > s} b(x,y) (u(x) - u(y))\\
  &< -1 + \frac{1}{m(x)} \sum_{y \in X, d(x,y) > s} b(x,y) u(y) \frac{d(x,y)^2}{s^2}\\
  &< -1 + \frac{\sup  u}{s^2}.
 \end{align*}
With this at hand stochastic completeness follows from the weak Omori-Yau principle. 
\end{proof}

\begin{remark}
 In principle the previous lemma is contained in \cite{GHM12}. Since the proof given there is a bit lengthy and since it is not straightforward to see that our Dirichlet form $Q$ and the Dirichlet form treated in \cite{GHM12} coincide, we chose to present a short proof based on the weak Omori-Yau principle. 
\end{remark}

%
%

We have now gathered all necessary reductions that allow us to prove our main theorem. The strategy is as follows. First we truncate the edge weight  and then we refine the truncated graph. For this refined graph we use the uniqueness class theorem and the volume growth assumption to obtain that bounded solutions to the heat equation with initial value $0$ are trivial. As mentioned after the definition of stochastically complete graphs, this is equivalent to stochastic completeness, see also \cite{KL12}.

\begin{proof}
	[Proof of Theorem~\ref{thm:main}]
		Recall that $\ls =\max\{\log, 1\}$.  Let $ b $ be a graph over $ (X,m) $  and let $ d $ be an intrinsic pseudo metric  with finite distance balls and
	\begin{align*}
	\int_{1}^{\infty}\frac{ r}{\ls({m(B_{r})})}dr=\infty.
	\end{align*}
	According to Lemma~\ref{lemma:truncation}, without loss of generality we can assume that $b$ has finite jump size with respect to $d$.  Finite jump size and finite distance balls imply that $b$ is locally finite, see \cite[Lemma~3.5]{Ke15}.

	Let $ n:X\times X\to\N_{0} $ be chosen such that the corresponding refinement $ b' $ over $(X',m')$ with pseudo metric $d'$ is globally local with respect to the logarithmic volume growth $f:(0,\infty) \to (0,\infty)$, $f(r) = \ls (m(B_r))$ which exists by Lemma~\ref{l:refinement}. According to Theorem~\ref{thm:refinements} it suffices to show that $b'$ over $(X',m')$ is stochastically complete. By \cite{KL10} this is equivalent to all bounded solutions to the heat equation with respect to  $b'$ with initial value $0$ being trivial. In order to verify this condition we employ Theorem~\ref{theorem:uniqueness class}.
	
	As above, we denote by $B_r'$ the distance ball of radius $r$ around $o$ with respect to $d'$. Lemma~\ref{lemma:properties of refinements} shows that $d'$ is intrinsic, that $d'$ has finite distance balls and that  $\ls (m'(B_r')) \leq f(r) + \log 2$.  Let now $u:(0,\infty) \times X' \to \mathbb R$ be a bounded solution to the heat equation.  For $T > 0$ it satisfies 
	$$ \int_0^T \sum_{x \in B_r'} |u_t(x)|^2 m(x) dt \leq T \sup u ^2m'(B'_r) \leq e^{f(r) + K},$$
	where $K > 0$ is a constant. Since $b'$ is globally local with respect to $f$ (and hence also with respect to $f + K$) and by assumption $f$ satisfies
	$$\int^\infty_1 \frac{r}{f(r) + K} dr = \infty,$$
	Theorem~\ref{theorem:uniqueness class} yields $u = 0$ on $(0,T) \times X'$. Since $T$ was arbitrary, we arrive at $u = 0$ and the claim is proven.
%
%
\end{proof}

\section{Examples and sharpness of the estimates} \label{section:examples}
	
	In this section we discuss an example which shows that our unique class criterion for globally local graphs is sharp in some sense. Already in the PhD thesis \cite{Hua11} there is a example on the integer line which shows that  without further assumptions on the graph one does not get the same uniqueness class statement as on manifolds. 
	
	Here we slightly modify and simplify this example to obtain a graph whose jump size $ s_{r} $ outside of balls decreases such that
	$$\frac{C^{-1}}{r} \leq   s_{r}\leq \frac{C}{r},\qquad r\ge 1, $$ 
	for some constant $ C\ge 1 $. 	According to Theorem~\ref{theorem:uniqueness class} such a graph does not allow non-trivial solutions to the heat equation whenever the function $ f $ in the growth condition \eqref{equation:growth condition 0} satisfies 
	$$ f(r)\leq C r^{2},\qquad r>1 ,  $$
	for some constant $ C > 0 $. However, such a graph is not globally local with respect to the function 
	$$   f(r) = C r^{2}\log r,\qquad r>1 ,$$ 
	where $ C\ge 1 $ is some constant  and we construct a non-trivial solution of the heat equation  that satisfies the growth  condition  \eqref{equation:growth condition 0} with respect to this  function. 
	
	This example underscores the following points. First of all, as was already observed in  \cite{Hua11}, the uniqueness class criterion from the case of manifolds does not hold  on graphs without further assumptions. Secondly, for globally local graphs our result improves the uniqueness class criterion of \cite{Huang12}. Thirdly, the globally local condition is sharp in the sense that if it is missed by a factor of a logarithm, then there are non-trivial solutions of the heat equation that satisfy the growth condition of Theorem~\ref{theorem:uniqueness class}.  	We discuss these points in detail while developing the example.

	Let $ X=\mathbb{Z} $ and $ m\equiv 1 $. Define a graph  over $ (X,m) $ by a symmetric $ b $  given by
	\begin{align*}
	b(0,-1)=1\qquad \mbox{and}\qquad		 b(n-1,n) = b(-n,-n-1)=n, \,n\ge 1 ,
	\end{align*}
	and $ b(k,l) =0$ for $ |k-l|\neq 1 $. Notice that $ d $ defined as $$  d (k,n)=\sum_{l=k}^{n}(1\vee (2l+1))^{-\frac{1}{2}},\qquad k,n\in \mathbb{Z},  $$
	is an intrinsic metric and there is $ C\ge0 $ such that 
	\begin{align*}
	C^{-1}d (0,n)^{2}	\leq n\leq C d(0,n)^2
	\end{align*}
	for all $ n\in \mathbb{Z} $. Therefore, we can compute $ s_{r} $
	\begin{align*}
	s_r 		 &=  \sup_{d(0,n)   \geq r}d(n,n+1) \geq C^{-1}\sup_{n \geq r^2} n^\frac{1}{2} \geq \frac{C^{-1}}{r}
	\end{align*}
	for some constant $ C\ge 1 $ and simultaneously $ s_{r}\leq C/r $ for $ r\ge 1 $.
	Our uniqueness class criterion now states that there are no non-trivial solutions $ u $ to the heat equation such that
	\begin{align*}
	\int_0^T \sum_{x \in B_r(0)} |u_t(x)|^2 m(x) dt \le   Ce^{c r^2},\qquad r> 0
	\end{align*}
	for certain  constants $ C,c>0 $. This improves \cite[Theorem~0.8]{Huang12},  which needs a bound of $C e^{c r\log r} $ with $0< c <1/2$ on the right hand side to guarantee triviality of $ u $.

	Next, we construct a non-trivial solution to the heat equation which satisfies the growth condition of our main theorem but  with respect to a function for which   the graph fails to be globally local by a factor of a logarithm.
	
	Let $ g:(0,\infty)\to (0,\infty) $
	\begin{align*}
	g(t)= e^{-t^{-2}}.
	\end{align*}
	We define $ u:\mathbb{Z}\times (0,\infty)\to (0,\infty)$ by
	\begin{align*}
	u_t(n)=\sum_{k=0}^{n}\frac{1}{k!}\binom{n}{k}g^{(k)}(t), \quad n\ge 0, t>0,
	\end{align*}
	and for $ n\leq -1, t>0 $ we let
	$$  u_t(n)=u_t(-n-1),$$
	i.e., $u$ is symmetric about the point $ 1/2 $ on the space axis. It is straight forward to verify
	\begin{align*}
	 \partial_{t}=-\mathcal{L}u.
	\end{align*}
	Furthermore, in order to estimate the growth of $ u $ we notice that  there is $ c>0$ such that for all $ t>0 $ and $ k\in\mathbb{N}_{0} $ the $ k $-th derivatives of $ g $ satisfy
	\begin{align*}
	|g^{(k)}(t)|\leq k!\left (\frac{2k}{c}\right)^{\frac{k}{2}},
	\end{align*}
	see \cite[Lemma~3.1]{Huang12} and the discussion preceeding it. This estimate and Stirling's formula imply that there is $ C\ge1  $ such that
	\begin{align*}
	|u_t(n)|\leq \sum_{k=0}^{n}\binom{n}{k} \left (\frac{2k}{c}\right)^{\frac{k}{2}}\le n! \left(\frac{2n}{c}\right)^{\frac{ n}{2}}\leq C e ^{C n\log n}.
	\end{align*}
	For $ r $ let $ n(r) $ be the largest $ n $ such that $ n\in B_{r}(0) $. We conclude that for $ 0\leq T<\infty  $ there are  constants $ C\ge 1 $ (which may change in every step of the inequality) such that
	\begin{align*}
	\int_0^T \sum_{x \in B_r(0)} |u_t(x)|^2 m(x) dt &\le Cn(r)|u_t(n(r))|^2\\
	&\leq C e^{C n(r)\log n(r)}\\
	&\leq C e^{C r^2\log r} =e^{f(r)},
	\end{align*}
	with 	$ f(r)= C r^2\log r +\log C$. This $ f $ satisfies
	$$  \int_{1}^\infty \frac{r}{f(r)}dr=\infty .$$
	So, while the function $ u $ satisfies the growth assumption of the uniqueness class theorem, Theorem~\ref{theorem:uniqueness class}, the graph fails to be globally local, as  $ s_{r}\ge C^{-1}/r $ and, therefore,
	$$   \frac{s_{r}f(A r)}{r}\ge \log r \to \infty,\qquad r\to\infty,$$
	for all constants $ A>1$.
	
	This shows that the globally local assumption is indeed needed for establishing a uniqueness class. Moreover, it shows that is sharp in some sense. Namely, we presented a counterexample where the globally local assumption is only missed by a factor of a logarithm.
	\bigskip

\textbf{Acknowledgements.} The authors use this opportunity to express their deep gratitude towards their former Ph.D. advisors Sascha Grigor'yan (X.H.) and Daniel Lenz (M.K., M.S.). We are grateful for every of the countless hours you discussed the topic of this paper with us. Furthermore, the authors M.K. and M.S. acknowledge the financial support of the DFG while X.H. is supported by National Natural Science Foundation of China(Grant No. 11601238), and The Startup Foundation for Introducing Talent of Nanjing University of Information Science and Technology (Grant No. 2015r053).
Moreover, the authors want to thank Bobo Hua and Fudan University for their hospitality during the conference ``Discrete Analysis'' when the major steps towards the proof of the results presented here were taken. 
\bibliographystyle{plain}
 
\bibliography{literature}

\def\cprime{$'$} \def\cprime{$'$}
\begin{thebibliography}{10}

\bibitem{Az74}
Robert Azencott.
\newblock Behavior of diffusion semi-groups at infinity.
\newblock {\em Bull. Soc. Math. France}, 102:193--240, 1974.

\bibitem{Da92}
E.~Brian Davies.
\newblock Heat kernel bounds, conservation of probability and the {F}eller
  property.
\newblock {\em J. Anal. Math.}, 58:99--119, 1992.
\newblock Festschrift on the occasion of the 70th birthday of Shmuel Agmon.

\bibitem{Fol14}
Matthew Folz.
\newblock Volume growth and stochastic completeness of graphs.
\newblock {\em Trans. Amer. Math. Soc.}, 366(4):2089--2119, 2014.

\bibitem{Ga59}
Matthew~P. Gaffney.
\newblock The conservation property of the heat equation on {R}iemannian
  manifolds.
\newblock {\em Comm. Pure Appl. Math.}, 12:1--11, 1959.

\bibitem{Gri86}
Alexander Grigor'yan.
\newblock Stochastically complete manifolds.
\newblock {\em Dokl. Akad. Nauk SSSR}, 290(3):534--537, 1986.

\bibitem{Grigoryan88}
Alexander Grigor'yan.
\newblock { Stochastically complete manifolds and summable harmonic functions
  (translation in Math. USSR-Izv. 33 (1989), no. 2, 425--432)}.
\newblock {\em Izv. Akad. Nauk SSSR Ser. Mat.}, 52(5):1102--1108, 1988.

\bibitem{Gri99}
Alexander Grigor'yan.
\newblock Analytic and geometric background of recurrence and non-explosion of
  the {B}rownian motion on {R}iemannian manifolds.
\newblock {\em Bull. Amer. Math. Soc. (N.S.)}, 36(2):135--249, 1999.

\bibitem{GHM12}
Alexander Grigor'yan, Xueping Huang, and Jun Masamune.
\newblock On stochastic completeness of jump processes.
\newblock {\em Math. Z.}, 271(3-4):1211--1239, 2012.

\bibitem{Hua11}
Xueping Huang.
\newblock {\em On stochastic completeness of graphs}.
\newblock 2011.
\newblock Ph.D. Thesis, Bielefeld.

\bibitem{Huang11}
Xueping Huang.
\newblock Stochastic incompleteness for graphs and weak {O}mori-{Y}au maximum
  principle.
\newblock {\em J. Math. Anal. Appl.}, 379(2):764--782, 2011.

\bibitem{Huang12}
Xueping Huang.
\newblock On uniqueness class for a heat equation on graphs.
\newblock {\em J. Math. Anal. Appl.}, 393(2):377--388, 2012.

\bibitem{Huang14}
Xueping Huang.
\newblock A note on the volume growth criterion for stochastic completeness of
  weighted graphs.
\newblock {\em Potential Anal.}, 40(2):117--142, 2014.

\bibitem{HKMW13}
Xueping Huang, Matthias Keller, Jun Masamune, and Rados{\l}aw~K. Wojciechowski.
\newblock A note on self-adjoint extensions of the {L}aplacian on weighted
  graphs.
\newblock {\em J. Funct. Anal.}, 265(8):1556--1578, 2013.

\bibitem{HS14}
Xueping Huang and Yuichi Shiozawa.
\newblock Upper escape rate of {M}arkov chains on weighted graphs.
\newblock {\em Stochastic Process. Appl.}, 124(1):317--347, 2014.

\bibitem{KarpLi}
Leon Karp and Peter Li.
\newblock { The heat equation on complete Riemannian manifolds}.
\newblock {\em unpublished}.

\bibitem{Ke15}
Matthias Keller.
\newblock Intrinsic metrics on graphs: a survey.
\newblock In {\em Mathematical technology of networks}, volume 128 of {\em
  Springer Proc. Math. Stat.}, pages 81--119. Springer, Cham, 2015.

\bibitem{KL10}
Matthias Keller and Daniel Lenz.
\newblock Unbounded {L}aplacians on graphs: basic spectral properties and the
  heat equation.
\newblock {\em Math. Model. Nat. Phenom.}, 5(4):198--224, 2010.

\bibitem{KL12}
Matthias Keller and Daniel Lenz.
\newblock Dirichlet forms and stochastic completeness of graphs and subgraphs.
\newblock {\em J. Reine Angew. Math.}, 666:189--223, 2012.

\bibitem{Stu94}
Karl-Theodor Sturm.
\newblock Analysis on local {D}irichlet spaces. {I}. {R}ecurrence,
  conservativeness and {$L^p$}-{L}iouville properties.
\newblock {\em J. Reine Angew. Math.}, 456:173--196, 1994.

\bibitem{Tak89}
Masayoshi Takeda.
\newblock On a martingale method for symmetric diffusion processes and its
  applications.
\newblock {\em Osaka J. Math.}, 26(3):605--623, 1989.

\bibitem{Woj08}
Rados{\l}aw~K. Wojciechowski.
\newblock {\em Stochastic completeness of graphs}.
\newblock ProQuest LLC, Ann Arbor, MI, 2008.
\newblock Thesis (Ph.D.)--City University of New York.

\bibitem{Woj11}
Rados{\l}aw~K. Wojciechowski.
\newblock Stochastically incomplete manifolds and graphs.
\newblock In {\em Random walks, boundaries and spectra}, volume~64 of {\em
  Progr. Probab.}, pages 163--179. Birkh\"auser/Springer Basel AG, Basel, 2011.

\end{thebibliography}

\end{document}